\documentclass[preprint 12pt]{elsarticle}
\usepackage{bbm}

\usepackage{amssymb}
\usepackage{amsfonts}
\usepackage{mathrsfs}
\usepackage{amssymb}
\usepackage{color, amsmath,amssymb, amsfonts, amstext,amsthm, latexsym}

\usepackage{amssymb, epsfig, amssymb, latexsym}
\usepackage{amsmath}
\usepackage{graphicx}
\usepackage{longtable}

\numberwithin{equation}{section}

\allowdisplaybreaks

\newtheorem{theorem}{Theorem}[section]
\newtheorem{lemma}{Lemma}[section]
\newtheorem{remark}{Remark}[section]

\newtheorem{proposition}{Proposition}[section]

\journal{Arxiv}

\begin{document}

\begin{frontmatter}



\title{Weak order in  averaging principle for stochastic wave
equations with  a fast oscillation }


\author{Hongbo Fu}
\ead{hbfuhust@gmail.com}
\author{Li Wan}
\ead{wanlinju@aliyun.com}
\address{Research Center of Nonlinear Science, College of Mathematics and Computer Science, Wuhan Textile University, Wuhan, 430073, PR China}

\author{Jicheng Liu}   
\ead{jcliu@hust.edu.cn}
\author{Xianming Liu\corref{cor1}}
\ead{xmliu@hust.edu.cn}
\address{School of Mathematics and Statistics, Huazhong University of Science and
Technology, Wuhan, 430074, PR China} 

\cortext[cor1]{Corresponding author at: School of Mathematics and
  Statistics, Huazhong University of Science and Technology, Wuhan,
      430074, China. }
\begin{abstract}

This article deals with the weak errors for averaging principle for
a stochastic wave equation in a bounded interval $[0,L]$, perturbed
by a oscillating term arising as the solution of a stochastic
reaction-diffusion equation evolving with respect to the fast time.
Under suitable conditions, it is proved that the rate of weak
convergence to the averaged effective dynamics is of order $1$ via
an asymptotic expansion approach.

\end{abstract}


\end{frontmatter}



 \section{Introduction}
Let $D=[0, L]\subset \mathbb{R}$ be a bounded open interval.  In the
article, for fixed $T>0$, we consider the following class of
stochastic wave equation with fast oscillating perturbation,
\begin{equation}\label{wave-equation}
\begin{cases}
\frac{\partial^2}{\partial t^2} U_t^{\epsilon}(\xi)=\Delta
U_t^{\epsilon}(\xi)+F(U_t^{\epsilon}(\xi), Y_\frac{t}{\epsilon}
(\xi))
+\sigma_1\dot{W_t^1}(\xi),\; t\in [0, T],\xi\in D, \\
U_t^{\epsilon}(\xi)=0, (\xi, t)\in \partial D\times (0, T], \\
U_0^{\epsilon}(\xi)=x_1(\xi),  \frac{\partial
U_t^{\epsilon}(\xi)}{\partial t}\Big|_{t=0}={x}_2(\xi), \xi\in D,
\end{cases}
\end{equation}
where $\epsilon$ is   positive parameter, $Y_t$ is governed by the
stochastic reaction-diffusion equation:
\begin{equation}\label{para-equation}
\begin{cases}
\frac{\partial }{\partial t}Y_t (\xi)= \Delta Y_t (\xi)+
 g( Y_t (\xi))+ {\sigma_2} \dot{W_t^2}(\xi),\; t\in [0, T],\xi\in D,\\
Y_t (\xi)=0, (\xi, t)\in \partial D\times (0, T], \\
Y_0 (\xi)=y(\xi),
\end{cases}
\end{equation}
Assumptions on the smoothness of the drift $f$ and $g$ will be given
below. The stochastic perturbations are of additive type and
$W^1_t(\xi)$ and $W^2_t(\xi)$ are mutually independent
$L^2(D)-$valued Wiener processes on a complete stochastic basis
$(\Omega, \mathscr{F}, \mathscr{F}_t, \mathbb{P})$, which will be
specified later. The noise strength coefficients $\sigma_1$ and
$\sigma_1$ are positive constants and the parameter $\epsilon$ is
small, which describes the ratio of time scale between the process
$X^\epsilon_t(\xi)$ and $Y_{t/\epsilon}(\xi)$.  With this time scale
the variable $X_t^\epsilon(\xi)$ is referred as slow component and
$Y_{t/\epsilon} (\xi)$ as the fast component.

The equation \eqref{wave-equation} is an abstract model for a random
vibration of a elastic string with  a fast oscillating perturbation.
More generally, the nonlinear coupled wave-heat equations with fast
and slow time scales may describe a thermoelastic wave propagation
in a random medium \cite{Chow-1}, the interactions of fluid motion
with other forms of waves \cite{Leung,X Zhang}, wave phenomena which
are heat generating or temperature related \cite{Leung-0},
magneto-elasticity \cite{Rivera} and biological problems \cite{Choi,
Cardetti, S.H. Wu}.

Averaging principle plays an important role in the study of
asymptotic behavior for slow-fast dynamical systems. It was first
studied by Bogoliubov\cite{Bogoliubov} for deterministic
differential equations. The theory of averaging for stochastic
ordinary equations may be found in \cite{Khas}, the works of
Freidlin and Wentzell \cite{Freidlin-Wentzell1, Freidlin-Wentzell2},
Veretennikov \cite{Vere1, Vere2}, and  Kifer \cite{Kifer1, Kifer2,
Kifer3}. Further progress on averaging for stochastic dynamical
systems  with non-Gaussian noise in finite dimensional space was
studied in \cite{Xu, Xu1, Xu2,Xu3,Xu4}. Concerning the infinite
dimensional case, it is worth quoting the paper by Cerrai
\cite{Cerrai1,Cerrai2,Cerrai-Siam}, Br\'{e}hier \cite{Brehier}, Wang
\cite{wangwei},  Fu \cite{Fu-Liu} and Bao \cite{Bao}.

In our  previous article \cite{Fu-Liu-2}, the asymptotic limit
dynamics (as $\epsilon$ tends to $0$) of system
\eqref{wave-equation} was explored within averaging
 framework. Under suitable conditions, it can be  shown that a reduced stochastic
wave equation, without the  fast component, can be constructed  to
characterize the essential  dynamics of \eqref{wave-equation} in a
pathwise sense,   as it is done in \cite{Cerrai1, Cerrai2,
Cerrai-Siam} for stochastic partial equations of parabolic type and
for stochastic ordinary differential equations \cite{Givon1, LiuDi,
Wainrib}.

In the present paper, we are interested in  the rate of weak
convergence of the averaging dynamics to the true solution of slow
motion $U^\epsilon_t(\xi)$. Namely, we will determine the order,
with respect to timescale parameter $\epsilon$, of weak deviations
between original solution of slow equation and  the solution  of the
corresponding averaged equation. To our knowledge, up to now this
problem has been treated only in the case of deterministic reaction
diffusion equations  in dimension $d=1$  subjected with a random
perturbation evolving with respect to the fast time $t/\epsilon$ (to
this purpose we refer to the  paper by Br\'{e}hier \cite{Brehier}).
Once the noise is included in slow variable, the method in
\cite{Brehier} used to obtain the weak order $1-\varepsilon$ for
arbitrarily small  $\varepsilon>0$ will be  more complicated due to
the lack of time regularity for slow solution.

In the situation we are considering, an additive time-space white
noise is included in the slow motion and the main results show that
order $1$ for weak convergence can be derived, which can  be
compared with the order $1-\varepsilon$ in \cite{Brehier}. Under
dissipative assumption on Eq. \eqref{para-equation}, the
perturbation process $Y_t$ admits a unique invariant measure $\mu$
with mixing property. Then, by averaging the drift coefficient  of
the slow motion Eq. \eqref{wave-equation} with respect to the
invariant measure $\mu$, the effective equation with following form
can be established:
\begin{equation*}
\begin{cases}
 \frac{\partial^2 }{\partial t^2}\bar{U}_t(\xi)=\Delta
\bar{U}_t(\xi)+\bar{F}(\bar{U}_t(\xi))+\sigma_1\dot{W_t^1}(\xi), \\
 \bar{U}_t(\xi)=0, (\xi, t)\in \partial D\times (0, T],\\
 \bar{U}_0(\xi)=x_1(\xi), \frac{\partial \bar{U}_t(\xi)}{\partial
t}\Big|_{t=0}= {x}_2(\xi), \xi\in D,
\end{cases}
\end{equation*}
where for any $u,y\in H:=L^2(D),$
\begin{equation*}
\bar{F}(u):=\int_HF(u,y)\mu(dy), u\in H.
\end{equation*}
We prove that, under a smoothness assumption on drift coefficient in
the slow motion equation, an error estimate of the following form
\begin{eqnarray*}
|\mathbb{E}\phi( {U}^\epsilon_t)-\mathbb{E}\phi(\bar{U}_t)|\leq
C\epsilon
\end{eqnarray*}
for any function $\phi$ with derivatives bounded up to order $3$. In
order to prove the validity of above bound, we adopt asymptotic
expansion schemes in \cite{Brehier} to decompose
$\mathbb{E}\phi({U}^\epsilon_t)$ with respect to the scale parameter
$\epsilon$ in form of
\begin{equation*}
\mathbb{E}\phi({U}^\epsilon_t)=u_0+\epsilon u_1+r^\epsilon,
\end{equation*}
where the functions $u_0$ has to coincide with
$\mathbb{E}\phi(\bar{U}_t)$ by  uniqueness discuss, as it can be
shown that they are governed by the same Kolmogorov equation via
identification the powers of $\epsilon$. Due to solvability of the
Poisson equation associated with generator of perturbation process
$Y_t$, an explicit expression of $u_1$ can be constructed such that
its boundedness is based on a priori estimates for the $Y_t$ and
smooth dependence on initial data for averaging equation. The next
step consist in identifying $r^\epsilon$ as the solution of a
evolutionary equation and showing that $|r^\epsilon|\leq C\epsilon$.
The proof of bound for $r^\epsilon$ is based on estimates on
$\frac{du_1}{dt}$ and $\mathcal{L}_2u_1$, where $\mathcal{L}_2 $ is
the  Kolmogorov operator associated  with the slow motion equation.
We would like to stress that this procedure is quite involved, as it
concerns a system with noise in infinite dimensional space, and the
diffusion term leading to quantitative analysis on higher order
differentiability of $\mathbb{E}\phi(\bar{U}_t)$ with respect to the
initial datum. Let us also remark that asymptotic expansion of the
solutions of Kolmogorov equations was studied in \cite{Khas2, Khas3}
and \cite{Weinan}.

The rest of the paper is arranged as follows.  Section 2 is devoted
to the general notation and framework. The ergodicity of fast
process and the averaging dynamics of system \eqref{wave-equation}
is introduced in Section 3. Then the main results of this article,
which is derived via the asymptotic expansions and uniform error
estimates, is presented in Section 4. In the final section, we state
and prove technical lemmas applied in the preceeding section.

Throughout the paper, the letter $C$ below with or without
subscripts will denote  generic positive constants independent of
$\epsilon$, whose value may  change from one line to another.

\section{Preliminary }
To rewrite  the systems \eqref{wave-equation} and
\eqref{para-equation} as the abstract evolution equations, we
present some notations and  some well-known facts for later use.

For a fixed domain $D=[0, L]$, we use the abbreviation $H:=L^2(D)$
for the space of square integrable real-valued functions on $D$. The
scalar product and norm on $H$ are denoted by $(\cdot, \cdot)_H$ and
$\|\cdot\|$, respectively.

We recall the definition of the Wiener process in infinite space.
For more details, see \cite{Daprato}. Let $\{q_{i,k}(\xi)\}_{k\in
\mathbb{N}}$ be $H$-valued eigenvectors of a nonnegative, symmetric
operator $Q_i$ with corresponding eigenvalues $\{\lambda_{i,
k}\}_{k\in \mathbb{N}}$, for $i=1, 2$, such that
\[Q_iq_{i,k}(\xi)=\lambda_{i, k} q_{i, k}(\xi),\; \lambda_{i, k}>0, k\in \mathbb{N}.\]
For $i=1, 2$, let $W_t^{i}(\xi)$ be an $H$-valued $Q_i$-Wiener
process with operator $Q_i$ satisfying
\begin{eqnarray*}
TrQ_i=\sum\limits_{k=1}^{+\infty}\lambda_{i, k}< {+\infty}.
\end{eqnarray*}
Then
\[W^{i}_t(\xi)=\sum\limits_{k=1}^{+\infty}{\lambda^\frac{1}{2}_{i, k}}\beta_{i, k}(t)q_{i,k}(\xi),\;t\geq 0,\]
where $\{\beta_{i, k}(t)\}^{i=1, 2}_{k\in\mathbb{N}}$ are mutually
independent real-valued Brownian motions on a probability base
$(\Omega, \mathscr{F}, \mathscr{F}_t, \mathbb{P})$. For the
abbreviation, we will sometimes omit the spatial variable $\xi$ in
the sequel.

Let $\{ e_k(\xi)\}_{k\in\mathbb{N}}$ denote the complete
orthornormal system of eigenfunctions in $H$ such that, for $k =
1,2,\ldots$,
\begin{eqnarray*}\label{eigenfunction} -\Delta
e_k=\alpha_ke_k,\;\;e_k(0)=e_k(L)=0 ,
\end{eqnarray*}
with $0<\alpha_1\leq\alpha_2\leq\cdots\alpha_k\leq\cdots$.
{Here we would like to recall the fact that
$e_k(\xi)=\sin\frac{k\pi\xi}{L}$ and
$\alpha_k=-\frac{k^2\pi^2}{L^2}$ for $k = 1,2,\cdots$.}

Let $A$ be the realization in $H$ of the Laplace operator $\Delta$
with zero Dirichlet boundary condition, which generates a strong
continuous semigroups $\{E_t\}_{t\geq 0},$ defined by, for any $h\in
H$,
\begin{eqnarray*}
E_th=\sum\limits_{k=1}^{+\infty} e^{-\alpha_kt}e_k\Big(e_k,
h\Big)_H.
\end{eqnarray*}
It is straightforward to check that $\{E_t\}_{t\geq0}$ are
contractive semigroups on $H$.

For $s\in \mathbb{R}$, we introduce the space $H^s:=D((-A)^{s/2})$,
which equipped with inner product
\begin{eqnarray*}
\langle
g,h\rangle_s:=\Big((-A)^{\frac{s}{2}}g,(-A)^{\frac{s}{2}}h\Big)_H=\sum\limits_{k=1}^{+\infty}\alpha^s_i\Big(g,e_k\Big)\Big(h,e_k\Big)_H,\;g,h\in
H^s
\end{eqnarray*}
and the norm
\begin{eqnarray*}
\|\varphi\|_s=\left\{\sum\limits_{k=1}^{+\infty}\alpha_k^s\Big(\varphi,
e_k\Big)_H^2\right\}^\frac{1}{2}
\end{eqnarray*}
for $\varphi\in H^s.$ It is obvious that $H^0=H$ and
$H^\alpha\subset H^\beta$ for $\beta\leq \alpha.$ We note that in
the case  of $s>0$, $H^{-s}$ can be identified with the dual space
$(H^{s})^*$, i.e. the space of the  linear functional on $H^{s}$
which are continuous with respect to the topology induced by the
norm $\|\cdot\|_s$. We shall denote by $\mathcal {H}^\alpha$ the
product space $H^\alpha\times H^{\alpha-1}, \alpha\in \mathbb{R}$,
endowed with the scalar product
\begin{eqnarray*}
\Big(x, y\Big)_{\mathcal {H}^\alpha}=\langle
x_1,y_1\rangle_\alpha+\langle x_2,y_2\rangle_{\alpha-1},\;\quad
x=(x_1, x_2)^T, y=(y_1, y_2)^T,
\end{eqnarray*}
and the corresponding norm
\begin{eqnarray*}
\||x|\|_\alpha=\left\{\|x_1\|^2_\alpha+\|x_2\|^2_{\alpha-1}\right\}^{\frac{1}{2}},\quad
x=(x_1,x_2)^T.
\end{eqnarray*}
If $\alpha=0$ we abbreviate $H^0\times H^{-1}=\mathcal{H}$ and
$\||\cdot\||=\||\cdot\||_0$. To consider \eqref{wave-equation} as an
abstract evolution equation, we set
$V^\epsilon_t=\frac{d}{dt}U^\epsilon_t$ and let $
X^\epsilon_t=\left[
\begin{array}{ccc}
 U^\epsilon_t \\
 V^\epsilon_t \\
 \end{array}
\right] $ with $X_0^\epsilon :=x=\left[
\begin{array}{ccc}
 x_1 \\
 x_2 \\
 \end{array}
\right]. $ The systems \eqref{wave-equation} and
\eqref{para-equation} can be rewritten as an abstract form
\begin{equation}\label{Abstract-equa}
\begin{cases}
 dX^\epsilon_t=\mathcal {A} X^\epsilon_tdt+\textbf{F}(X^\epsilon_t,Y^\epsilon_t)dt+BdW^1_t,\\
dY^\epsilon_t=\frac{1}{\epsilon}{A} Y^\epsilon_tdt+\frac{1}{\epsilon}g(Y^\epsilon_t)dt+\frac{\sigma}{\sqrt{\epsilon}}dW^2_t, \\
X^\epsilon_0=x, Y^\epsilon_0=y,
\end{cases}
\end{equation}
where
\begin{equation*}
\mathcal {A}:=\left[
\begin{array}{ccc}
 0 & I \\
 A & 0
\end{array}
\right], \textbf{F}(x, y):=\left[
\begin{array}{ccc}
 0  \\
 F(\Pi_1\circ x,y)
\end{array}
\right], B:=\left[
\begin{array}{ccc}
 0  \\
 I
\end{array}
\right],
\end{equation*}
with
$$\mathcal {D}(\mathcal{A})=\left\{X=(x_1,x_2)^T\in \mathcal {H}:\mathcal{A}X=\left[
\begin{array}{ccc}
 x_2  \\
 Ax_1
\end{array}
\right]\in \mathcal {H}\right\}=\mathcal {H}^1,$$ here $A$ is
regarded as an operator from $H^1$ to $H^{-1}$, and $\Pi_1$ denotes
the canonical projection $\mathcal {H} \rightarrow H $. It is well
known that the operator $\mathcal{A}$ is the generator of a strongly
continuous semigroup $\{\mathcal {S}_t\}_{t\geq0}$ on $\mathcal{H} $
with the explicit form
\begin{eqnarray}\label{macal-S-t}
\mathcal {S}_t=e^{\mathcal{A}t}=\left[
\begin{array}{ccc}
 C(t) &  (-A)^{-\frac{1}{2}}S(t) \\
 - (-A)^{\frac{1}{2}}S(t) & C(t)
\end{array}
\right], \;t\geq 0,
\end{eqnarray}
where $C(t)=\cos({(-A)^\frac{1}{2}} t)$ and
$S(t)=\sin({(-A)^\frac{1}{2}} t)$ are so-called cosine and sine
operators  with the expression in
 term of the orthonormal eigenpairs $\{\alpha_1, e_i\}_{i\in
  \mathbb{N}}$ of $A$:

\begin{eqnarray*}
  &&C(t)h=\cos({(-A)}^\frac{1}{2} t)h=\sum\limits_{k=1}^{+\infty}
    {\cos\{\sqrt{\alpha_k}t\}} \Big(e_k,
  h\Big)_H\cdot e_k,\\
 && S(t)h=\sin({(-A)}^\frac{1}{2} t)h=\sum\limits_{k=1}^{+\infty}
    {\sin\{\sqrt{\alpha_k}t\}} \Big(e_k,
  h\Big)_H\cdot e_k.
\end{eqnarray*}
Moreover, it is easy to check that $\||\mathcal{S}_t
x\||\leq\||x\||$ for $t\geq 0, x\in \mathcal{H}.$ In order to ensure
existence and uniqueness of the perturbation process $Y_t$ we shall
assume throughout this paper that:\\
(\textbf{Hypothesis 1}) For the mapping $g:H \rightarrow H$, we
require that there exists a constant $L_g>0$ such that
\begin{eqnarray}
\|g(u_1)-g(u_2)\|\leq L_g(\|u_1-u_2\|), \;u, v\in H\label{g-condi}.
\end{eqnarray}
moreover, we assume that $L_g< \alpha_1.$

Concerning the coefficient $F$ we impose the following
conditions:\\
(\textbf{Hypothesis 2}) For the mapping $F:H\times H \rightarrow H$,
we assume that there exists a constant $L_F>0$ such that
\begin{eqnarray}
\|F(u_1,v_1)-F(u_2,v_2)\|\leq
L_F(\|u_1-u_2\|+\|v_1-v_2\|),\;u_1,u_2, v_1,v_2\in
H.\label{F-condi-1}
\end{eqnarray}
Also suppose that for any $u\in H$, the mapping $F(u,\cdot):
H\rightarrow H$ is of class $C^2$, with bounded derivatives.
Moreover, we require that there exists a constant $L$ such that for
any $u,v,w,y,y'\in H$ its directional derivatives are well-defined
and satisfy
\begin{eqnarray}
&&\|D_uF(u,y)\cdot w\|\leq L \|w\|,\label{F'}\\
&&\|D^2_{uu}F(u,y)\cdot (v,w)\|\leq L \|v\|\cdot\|w\|.\label{F''}\\
&&\|[D_uF(u,y)-D_uF(u,y')]\cdot w\|\leq L \|y-y'\|\cdot\|w\|\label{F'-lip}\\
&&\|D^2_{uu}[F(u,y)-F(u,y')]\cdot (v,w)\|\leq L
\|y-y'\|\cdot\|v\|\cdot\|w\|.\label{F''-lip}
\end{eqnarray}

\begin{remark}
A simple example of the dirft coefficient  $F$ is given by
\begin{equation*}
F(u,y)=F_1(u)+F_2(y),
\end{equation*}
here $F_1, F_2: H\rightarrow H$ are of class $C^2$ with   uniformly
bounded derivatives up to order 2.
\end{remark}

According to conditions \eqref{g-condi} and \eqref{F-condi-1},
system \eqref{Abstract-equa} admits a unique mild solution. Namely,
as discussed in \cite{Daprato}, for  any  $y\in H$ there exists a
unique adapted process $Y (y)\in L^2(\Omega,C([0, T];H)$ such that

\begin{eqnarray}
Y _t(y)=E_{t }y + \int_0^tE_{t-s }g( Y_s (y))ds+ {\sigma_2}
\int_0^tE_{t-s }dW^2_s,\label{Integral-Fast}
\end{eqnarray}
By arguing as in the proof of \cite{Daprato}, Theorem 7.2,  it is
possible to show that there exists a constant $C>0$ such that
\begin{eqnarray}
\mathbb{E}\|Y_t(y)\|^2\leq C(1+\|y\|^2),\;t>0, \label{Y-bound}
\end{eqnarray}
and in correspondence of such $Y_{t}(y)$, for any $\epsilon>0$ and
$x=(x_1, x_2)^T\in \mathcal{H}$ there exists a unique adapted
process $X^\epsilon(x,y)\in L^2(\Omega,C([0,T ];\mathcal{H}))$ such
that
\begin{eqnarray}
X^\epsilon_t(x,y)=\mathcal {S}_tx+\int_0^t\mathcal
{S}_{t-s}\textbf{F}(X_s^{\epsilon}(x,y), Y
_{s/\epsilon}(y))ds+\sigma_1\int_0^t\mathcal
{S}_{t-s}BdW^1_s.\label{Integral-Slow}
\end{eqnarray}
We point out   that if $x=(x_1,x_2)^T$ is taken in
$D(\mathcal{A})=\mathcal{H}^1$, then $X^\epsilon_t$ values in $
\mathcal{H}^1$ for $t>0$ (see \cite{Chow}) and satisfies
\begin{eqnarray}
\mathbb{E}\||X^\epsilon_t(x,y)\||^2_1\leq C(1+\|y\|^2+\|x\|^2_1)
\label{2.6}
\end{eqnarray}
for some constant $C>0$. Moreover, we present an estimate for the
$\mathcal{H}-$norm of $\mathcal{A}X_t^\epsilon$, which is uniform
with respect to $\epsilon>0$.
\begin{proposition}
Let ${X}^\epsilon_t(x, y)=(U_t^\epsilon(x, y), V_t^\epsilon(x,
y))^T$  be the solution to the problem \eqref{Integral-Slow}, where
the initial value  satisfies ${X}^\epsilon_0=x=(x_1, x_2)^T\in
\mathcal{H}^1$, and the function  $F$ satisfies \eqref{F-condi-1}.
Then it holds that
\begin{eqnarray}
\mathbb{E}\||\mathcal{A}X^\epsilon_t(x, y)\||^2\leq C
(1+\|y\|^2+\||x\||_1^2).\label{AX-bound}
\end{eqnarray}
\begin{proof}
We have
\begin{equation*}
\mathcal{A}X^\epsilon_t(x,y)=\left(
\begin{array}{ccc}
 0 & I   \\
 A & 0
 \end{array}
\right) \left(
\begin{array}{ccc}
 U_t^\epsilon(x,y)  \\
 V_t^\epsilon(x,y)
 \end{array}
\right)=
 \left(
\begin{array}{ccc}
 V_t^\epsilon(x,y)  \\
 A(U_t^\epsilon(x,y))
 \end{array}
\right),
\end{equation*}
so that
\begin{eqnarray}
\||\mathcal{A}X^\epsilon_t(x,
y)\||^2&=&\|V_t^\epsilon(x,y)\|^2+\|A(U_t^\epsilon(x,y))\|^2_{-1}\nonumber\\
&=&\|V_t^\epsilon(x,y)\|^2+\|A^{\frac{1}{2}}(U_t^\epsilon(x,y))\|^2
. \label{2.7}
\end{eqnarray}
Let us start to  estimate  the norm of
$A^{\frac{1}{2}}(U_t^\epsilon(x,y))$ and consider the expression
\begin{eqnarray*}
A^{\frac{1}{2}}U_t^\epsilon(x,y)&=&A^{\frac{1}{2}}C(t)x_1-S(t)x_2-\int_0^tS(t-s)F(U_s^\epsilon(x,y),Y _{s/\epsilon}(y))ds\\
&+&\sigma_1\int_0^tS(t-s)dW^1_s.
\end{eqnarray*}
Directly, we have
\begin{eqnarray}
\|A^{\frac{1}{2}}C(t)x_1\|^2+\|S(t)x_2\|^2 \leq
C(\|x_1\|^2_1+\|x_2\|^2). \label{2.8}
\end{eqnarray}
In view of the assumptions on $F$ given in \eqref{F-condi-1}, we
obtain
\begin{eqnarray}
&&\mathbb{E}\|\int_0^tS(t-s)F(U_s^\epsilon(x,y),Y_{s/\epsilon}(y))ds\|^2\nonumber\\
&&\leq C_1+C_2\int_0^t\mathbb{E}
[\|U_s^\epsilon(x,y)\|^2+\|Y_{s/\epsilon}(y)\|^2]ds\nonumber\\
&&\leq C_1+C_2\int_0^t\mathbb{E}
[\|A^{\frac{1}{2}}U_s^\epsilon(x,y)\|^2+\|Y_{s/\epsilon}(y))\|^2]ds,\nonumber
\end{eqnarray}
and then, thanks to \eqref{Y-bound}, we have
\begin{eqnarray}
&&\mathbb{E}\|\int_0^tS(t-s)F(U_s^\epsilon(x,y),Y_{s/\epsilon}(y))ds\|\nonumber\\
&&\leq C_1(1+\|y\|^2)+C_2\int_0^t
\mathbb{E}\|A^{\frac{1}{2}}U_s^\epsilon(x,y)\|^2ds. \label{2.9}
\end{eqnarray}
Notice that in view of Ito's isometry,  we have
\begin{eqnarray*}
\mathbb{E}\|\int_0^tS(t-s)dW^1_s\|^2\leq C_3
\end{eqnarray*}
and then, combining this estimate with \eqref{2.8} and \eqref{2.9},
we have
\begin{eqnarray}
\mathbb{E}\|A^{\frac{1}{2}}U_t^\epsilon(x,y)\|^2&\leq&
C_1(1+\|y\|^2+\|x_1\|_1^2+\|x_2\|^2)\nonumber\\
&+&C_2\int_0^t\mathbb{E}
\|A^{\frac{1}{2}}U_s^\epsilon(x,y)\|^2ds.\nonumber
\end{eqnarray}
From the Gronwall's lemma, this gives
\begin{eqnarray*}
\mathbb{E}\|A^{\frac{1}{2}}U_t^\epsilon(x,y)\|^2&\leq&
C_1(1+\|y\|^2+\|x_1\|_1^2+\|x_2\|^2).
\end{eqnarray*}
In an analogous way,  we can prove that
\begin{eqnarray*}
\mathbb{E}\|V_t^\epsilon(x,y)\|^2\leq
C_1(1+\|y\|^2+\|x_1\|_1^2+\|x_2\|^2).
\end{eqnarray*}
Thanks to \eqref{2.7}, the two inequalities above yield
\eqref{AX-bound}.
\end{proof}
\end{proposition}


If $\mathcal {X}$ is a Hilbert space equipped with inner product
$(\cdot,\cdot)_\mathcal {X}$, we denote by $C^1(\mathcal
{X},\mathbb{R})$ the space of all real function $\phi:\mathcal
{X}\rightarrow \mathbb{R}$ with continuous Fr\'{e}chet derivative
and  use the notation $D\phi(x)$ for the differential of a $C^1$
function on $\mathcal {X}$ at the point $x$. Thanks to Riesz
representation theorem, we may get the identity for $x,h\in \mathcal
{X}$:
$$D\phi(x)\cdot h=(D\phi(x), h)_\mathcal {X}.$$
We define $C_b^2(\mathcal {X}, \mathbb{R})$ to be the space of all
real-valued, twice Fr\'{e}chet differential function on $\mathcal
{X}$, whose first and second derivatives are continuous and bounded.
For $\phi\in C_b^2(\mathcal {X}, \mathbb{R})$, we will identify
$D^2\phi(x)$ with a bilinear operator from $\mathcal {X}\times
\mathcal {X}$ to $\mathbb{R}$ such that
$$D^2\phi(x)\cdot (h,k)=(D^2\phi(x)h,k)_\mathcal {X},\;\; x,h,k\in \mathcal {X}.$$
On some occasions, we also use the notation $\phi',\phi''$ instead
of $D\phi$ or $D^2\phi.$

\section{Ergodicity of $Y_t$ and averaging dynamics}
Now, we consider the transition semigroup $P_t$ associated with
perturbation process $Y_t(y)$ defined by equation
\eqref{Integral-Fast}, by setting for any $\psi \in \mathcal
{B}_b(H)$ the space of bounded functions on $H$,
\begin{equation*}
P_t\psi(y)=\mathbb{E}\psi(Y_t(y)).
\end{equation*}
By arguing as \cite{Fu-Liu}, we can show that
\begin{equation}\label{Fast-motion-energy-bound}
\mathbb{E}\|Y_t(y)\|^2\leq
C\left(e^{-(\alpha_1-L_g)t}\|y\|^2+1\right),\; t>0
\end{equation}
for some constant $C>0$. This implies that there exists an invariant
measure $\mu $ for the Markov semigroup $P _t$ associated with
system \eqref{Integral-Fast} in $H$ such that
$$ \int_HP_t \psi d\mu =\int_H\psi d\mu , \quad
t\geq 0
$$
for any  $\psi \in \mathcal {B}_b(H)$ (for a proof, see, e.g.,
\cite{Cerrai2}, Section 2.1). Then by repeating the standard
argument as in the proof of Proposition 4.2 in \cite{Cerrai-Siam},
the invariant measure has finite $2-$moments:
\begin{equation}\label{mu-Momenent-bound}
\int_H\|y\|^2\mu(dy)\leq C.
\end{equation}
Let $Y_t(y')$ be the solution of  \eqref{Integral-Fast} with initial
value $Y_0=y'$, it can be check  that for any $t\geq0$,
\begin{eqnarray}\label{initial-diff}
\mathbb{E}\|Y_t(y)-Y_t( y')\|^2\leq\|y-y'\|^2e^{-\eta t}
\end{eqnarray}
with $\eta=(\alpha_1-L_g)>0,$ which implies that $\mu $ is the
unique invariant measure for $P _t$. Then, by averaging the
coefficient $F$ with respect to the invariant measure $\mu$, we can
define a $H$-valued mapping
\begin{equation*}
\bar{F}(u):=\int_HF(u,y)\mu(dy), u\in H, \label{aver-F}
\end{equation*}
and then, due to condition \eqref{F-condi-1}, it is easily to check
that
\begin{eqnarray}
\|\bar{F}(u_1)-\bar{F}(u_2)\|\leq L\|u_1-u_2\|, \; u_1, u_2\in H.
\label{barF-lip}
\end{eqnarray}

Now we  will consider the  effective dynamics system
\begin{eqnarray}
 \label{Averaging-equation}
 \begin{cases}
 \frac{\partial^2}{\partial
t^2}\bar{U}_t(\xi)=\Delta \bar{U}_t(\xi)+ \bar{F}(\bar{U}_t(\xi))+\sigma_1\dot{W}_t^{1},\;(\xi,t)\in D\times[0, T],\\
 \bar{U}_t(\xi)=0, (\xi, t)\in \partial
D\times [0, {+\infty}), \\
 \bar{U}_0(\xi)=x_1(\xi), \frac{\partial}{\partial
t} \bar{U}_t(\xi)|_{t=0}=x_2(\xi), \;\xi\in D.
\end{cases}
\end{eqnarray}
Following the same notation as in Section 2,  the problem
\eqref{Averaging-equation} can be transferred to a stochastic
evolution equation:
\begin{equation}\label{Abstract-aver-equa}
\begin{cases}
 d\bar{X}_t=\mathcal {A} \bar{X}_tdt+ \bar{\bf{F}}(\bar{X}_t)dt+BdW^1_t,\\
X_0=x,
\end{cases}
\end{equation}
where
 $
\bar{X}_t=\left[
\begin{array}{ccc}
 \bar{U}_t \\
 \bar{V}_t \\
 \end{array}
\right]$ with $\bar{V}_t=\frac{d}{dt}\bar{U}_t$ and
$\bar{\bf{F}}(x):=\left[
\begin{array}{ccc}
 0  \\
 \bar{F}(\Pi_1\circ x)
\end{array}
\right]=\left[
\begin{array}{ccc}
 0  \\
 \bar{F}(u)
\end{array}
\right].$ The mild form for system
 \eqref{Abstract-aver-equa} is given by
\begin{eqnarray*}
\bar{X}_t(x)=\mathcal {S}_tx+\int_0^t\mathcal
{S}_{t-s}\bar{\bf{F}}(\bar{X}_s(x))ds+\sigma_1\int_0^t\mathcal
{S}_{t-s}BdW^1_s.\label{Aver-Integral}
\end{eqnarray*}
By arguing as before, for any $x=(x_1, x_2)^T\in \mathcal{H}$ the
above integral equation admits a unique mild solution in $
L^2(\Omega,C([0,T ];\mathcal{H}))$ such that
\begin{eqnarray}
\mathbb{E}\||\bar{X}_t(x)\||\leq C(1+\||x\||),\;t\in [0,
T].\label{bar-X-bound}
\end{eqnarray}

\section{Asymptotic expansions}

Let $\phi\in C_b^2( {H} , \mathbb{R})$ and define a function
$u^\epsilon: [0, T]\times \mathcal {H} \times H\rightarrow
\mathbb{R}$ by
$$u^\epsilon(t, x,y)=\mathbb{E}\phi(U_t^\epsilon(x,y)).$$
Let $\Pi_1$ be the canonical projection $\mathcal {H} \rightarrow H
$. Define the function $\Phi: \mathcal{H} \rightarrow \mathbb{R}$ by
$\Phi(x):=\phi(\Pi_1  x)=\phi(x_1)$ for $x=(x_1,
x_2)^T\in\mathcal{H} $. Clearly, we have
\begin{eqnarray*}
u^\epsilon(t, x,y)=\mathbb{E}\Phi(X_t^\epsilon(x,y)).
\end{eqnarray*}
 We now introduce  two differential operators associated with the
   systems \eqref{Integral-Fast} and \eqref{Integral-Slow}, respectively:
\begin{eqnarray*}
\mathcal {L}_1\varphi(y)&=&\Big(Ay+g(y),
D_y\varphi(y)\Big)_H\\
&&+\frac{1}{2}\sigma_2^2Tr(D^2_{yy}\varphi(y)Q_2(Q_2)^*),\;
\varphi(y)\in C_b^2(H,\mathbb{R}),
\end{eqnarray*}
\begin{eqnarray*}
\mathcal {L}_2\Psi(x)&=&\Big(\mathcal {A}x+\textbf{F}(x,y),
D_x\Psi(x)\Big)_{\mathcal {H}
}\\
&&+\frac{1}{2}\sigma_1^2Tr(D^2_{xx}\Psi(x)BQ_1(BQ_1)^*),\;
\Psi(x)\in C_b^2(\mathcal {H},\mathbb{R} ).
\end{eqnarray*}
It is known that $u^\epsilon$ is a solution to the forward
Kolmogorov equation:
\begin{eqnarray}\label{Kolm}
\begin{cases}
\frac{d}{dt}u^\epsilon(t, x, y)=\mathcal {L}^\epsilon u^\epsilon(t, x, y),\\
u^\epsilon(0, x,y)=\Phi(x),
\end{cases}
\end{eqnarray}
where $\mathcal {L}^\epsilon=\frac{1}{\epsilon}\mathcal
{L}_1+\mathcal {L}_2.$

 Also recall the Kolmogorov operator for the averaging system is
defined as
\begin{eqnarray*}
\bar{\mathcal {L}}\Psi(x)&=&\Big(\mathcal {A}x+\bar{\textbf{F}}(x),
D_x\Psi(x)\Big)_{\mathcal {H}
}\\
&&+\frac{1}{2}\sigma_1^2Tr(D^2_{xx}\Psi(x)BQ_1(BQ_1)^*),\;
\Psi(x)\in C_b^2(\mathcal {H}, \mathbb{R} ).
\end{eqnarray*}
If we set $$\bar{u}(t,
x)=\mathbb{E}\phi(\bar{U}_t(x))=\mathbb{E}\Phi(\bar{X}_t(x)),$$ we
have
\begin{eqnarray}\label{Kolm-Aver}
\begin{cases}
\frac{d}{dt}\bar{u}(t, x)=\bar{\mathcal {L}} \bar{u}(t, x),\\
\bar{u}(0, x)=\Phi(x).
\end{cases}
\end{eqnarray}
Then the weak difference at time $T$ is equal to
\begin{eqnarray*}
\mathbb{E}\phi(\bar{U}_T)-\mathbb{E}\phi({U}^\epsilon_T)=u^\epsilon(T,
x,y)-\bar{u}(T,x).
\end{eqnarray*}
Henceforth, when there is no confusion, we often omit the temporal
variable $t$ and spatial variables $x$ and $y$. For example, for
$u^\epsilon(t, x, y)$, we often write it as $u^\epsilon$. Our aim is
to seek matched asymptotic expansions for the $u^\epsilon(T, x,y)$
of the form

\begin{eqnarray}\label{asymp-expan}
u^\epsilon=u_0+\epsilon u_1+r^\epsilon,
\end{eqnarray}
where  $u_0$ and $u_1$ are smooth functions which will be
constructed below , and $r^\epsilon$ is the remainder term. With the
above assumptions and notation we have the following   result, which
is a direct consequence of Lemma \ref{u-0}, Lemma
\ref{u-1-abso-lemma} and Lemma \ref{4.5}.
\begin{theorem}
Assume that $x\in \mathcal{H}^1, y\in D(A).$ Then, under Hypotheses
1
 and 2, for any any $T>0$ and $\phi\in C_b^3(H)$, there exist a
 constant $C_{T,\phi,x,y}$ such that
 \begin{eqnarray*}
 \left|\mathbb{E}\phi(U^\epsilon_T(x,y))-\mathbb{E}\phi(\bar{U}_T(x))\right|\leq
 C_{T,\phi,x,y}\epsilon.
 \end{eqnarray*}
\end{theorem}

\subsection{\textbf{The leading term}}
Let us first   determine the leading terms. Now, substituting
\eqref{asymp-expan} into \eqref{Kolm} yields
\begin{eqnarray*}
\frac{du_0}{dt}+\epsilon\frac{du_1}{dt}+\frac{dr^\epsilon}{dt}&=&
\frac{1}{\epsilon}\mathcal {L}_1u_0+\mathcal {\mathcal
{L}}_1u_1+\frac{1}{\epsilon}\mathcal {L}_1r^\epsilon\\
&&+\mathcal {L}_2u_0+\epsilon \mathcal {L}_2u_1+\mathcal
{L}_2r^\epsilon.
\end{eqnarray*}
By comparing coefficients of  powers of $\epsilon$, we obtain
\begin{eqnarray}
&&\mathcal {L}_1u_0=0, \label{u-o-equ-1} \\
&&\frac{du_0}{dt}=\mathcal {L}_1u_1+\mathcal
{L}_2u_0.\label{u-0-equ-2}
\end{eqnarray}
It follows from \eqref{u-o-equ-1} that $u_0$ is independent of $y$,
which means $$u_0(t,x, y)=u_0(t,x).$$ We also impose the initial
condition $u_0(0,x)=\Phi(x).$ Since $\mu$ is the invariant measure
of a Markov process with generator $\mathcal {L}_1$, we have
\begin{eqnarray*}
\int_H\mathcal {L}_1u_1(t,x,y)\mu(dy)=0,
\end{eqnarray*}
which, by invoking \eqref{u-0-equ-2}, implies
\begin{eqnarray*}
\frac{du_0}{dt}(t,x)&=&\int_H\frac{du_0}{dt}(t,x)\mu(dy)\\
&=&\int_H\mathcal {L}_2u_0(t,x)\mu(dy)\\
&=&\left(\mathcal {A}u_0(t,x)+\int_H\textbf{F}(x,y)\mu(dy),
D_xu_0(t,x)\right)_{\mathcal
{H}}\nonumber\\
&&+\frac{1}{2}\sigma_1^2Tr(D^2_{xx}u_0(t,x)BQ_1(BQ_1)^*)\\
&=&\bar{\mathcal {L}}u_0(t,x),
\end{eqnarray*}
so that $u_0$ and $\bar{u}$ satisfies the same evolution equation.
By using   a uniqueness argument, such $u_0$ has to coincide with
the solution $\bar{u}$  and we have the following lemma:
\begin{lemma}\label{u-0}
Assume Hypotheses 1 and 2. Then for any $x\in D(\mathcal{A})$, $y\in
D(A)$ and $T>0$, we have $u_0(T,x,y)=\bar{u}(T,x)$.
\end{lemma}

\subsection{\textbf{Construction of} $ {u_1}$}
Let us proceed to carry out the construction of $u_1$. Thanks to
Lemma \ref{u-0} and \eqref{Kolm-Aver}, the equation
\eqref{u-0-equ-2} can be rewritten
\begin{eqnarray*}
\bar{\mathcal {L}}\bar{u}=\mathcal{L}_1u_1+\mathcal{L}_2\bar{u},
\end{eqnarray*}
and hence we get an elliptic equation for $u_1$ with form
\begin{eqnarray*}
\mathcal {L}_1u_1(t,x,y)=\Big(\bar{\bf{F}}(x)- \textbf{F}(x,y),
D_x\bar{u}(t,x)\Big)_{\mathcal {H} }:=-\rho(t,x,y),
\end{eqnarray*}
where $\rho$ is of class $C^2$ with respect to $y$, with uniformly
bounded derivative. Moreover, it satisfies for any $t\geq 0$ and
$x\in \mathcal{H}^1$,
\begin{eqnarray*}
\int_H\rho(t,x,y)\mu(dy)=0.
\end{eqnarray*}
For any $y\in D(A)$ and $s>0$ we have
\begin{eqnarray*}
\frac{d}{ds}P_s\rho(t,x,y)&=&\Big(Ay+g(x,y),D_y(P_s\rho(t,x,y))\Big)_H\nonumber\\
&+&\frac{1}{2}\sigma_2^2Tr[D^2_{yy}(P_s\rho(t,x,y))Q_2Q_2^*],
\end{eqnarray*}
here $$P_s\rho(t,x,y)=\mathbb{E}\rho(t, x,Y_s(y))$$ satisfying
\begin{equation}
\lim\limits_{s\rightarrow{+\infty}}\mathbb{E}\rho(t,
x,Y_s(y))=\int_H\rho(t,x,z)\mu(dz)=0.\label{rho-limits}
\end{equation}   Indeed, by
the invariant property of $\mu$ and Lemma \ref{ux},
\begin{eqnarray*}
&&\left|\mathbb{E}\rho(t,
x,Y_s(y))-\int_H\rho(t,x,z)\mu(dz)\right|\nonumber\\
&&=\left|\int_H\mathbb{E}[\rho(t, x,Y_s(y))- \rho(t,
x,Y_s(z))\mu(dz)]\right|\nonumber\\
&&\leq\int_H\left|\mathbb{E}\Big(\textbf{F}(x,Y_s(z)-
\textbf{F}(x,Y_s(y), D_x\bar{u}(t,x)\Big)_{\mathcal {H}
}\right|\mu(dz)\nonumber\\
&&\leq C\int_H \mathbb{E} \|Y_s(z)- Y_s(y)\|
\mu(dz).\nonumber\\
\end{eqnarray*}
This, in view of \eqref{initial-diff} and \eqref{mu-Momenent-bound},
yields
\begin{eqnarray*}
&&\left|\mathbb{E}\rho(t,
x,Y_s(y))-\int_H\rho(t,x,z)\mu(dz)\right|\nonumber\\
&&\leq Ce^{-\frac{\eta}{2}s},
\end{eqnarray*}
which implies the equality \eqref{rho-limits}. Therefore, we get
\begin{eqnarray}
&&\Big(Ay+g(x,y),D_y\int_0^{{+\infty}}P_s\rho(t,x,y)
ds\Big)_H\nonumber\\
&&+\frac{1}{2}\sigma_2^2Tr[D^2_{yy}\int_0^{{+\infty}}(P_s\rho(t,x,y))Q_2Q_2^*]ds\nonumber\\
&&=\int_0^{{+\infty}}\frac{d}{ds}P_s\rho(t,x,y)ds\nonumber\\
&&=\lim\limits_{s\rightarrow{+\infty}}\mathbb{E}\rho(t,
x,Y_s(y))-\rho(t,x,y)\nonumber\\
&&=\int_H\rho(t,x,y)\mu(dy)-\rho(t,x,y)\nonumber\\
&&=-\rho(t,x,y),\nonumber
\end{eqnarray}
which means $\mathcal{L}_1(\int_0^{{+\infty}}P_s\rho(t,x,y)
ds)=-\rho(t,x,y).$ Therefore, we can set
\begin{eqnarray}\label{u-1}
u_1(t,x,y)=\int_0^{+\infty} \mathbb{E}\rho(t,x,Y_s(y))ds.
\end{eqnarray}
\begin{lemma}\label{u-1-abso-lemma}
Assume Hypotheses 1  and 2. Then for any $x\in D(\mathcal{A})$,
$y\in D(A)$ and $T>0$, we have
\begin{eqnarray}
|u_1(t,x,y)|\leq C_T(1 +\|y\|),\;t\in[0, T]. \label{u-1-abso}
\end{eqnarray}
\begin{proof}
As known from \eqref{u-1}, we have
\begin{equation*}
u_1(t,x,y)=\int_0^{{+\infty}}\mathbb{E}\Big(\bar{\bf{F}}(x)-
\textbf{F}(x,Y_s(y)), D_x\bar{u}(t,x)\Big)_{\mathcal {H} }ds.
\end{equation*}
This implies that
\begin{eqnarray*}
|u_1(t,x,y)|&\leq&\int_0^{{+\infty}}\||\bar{\bf{F}}(x)-
\mathbb{E}[\textbf{F}(x,Y_s(y))]\||  \cdot\||D_x\bar{u}(t,x)\|| ds.
\end{eqnarray*}
Then, in view of Lemma \ref{ux} and \eqref{Averaging-Expectation},
this implies :
\begin{eqnarray*}
 |u_1(t,x,y)|&\leq& C_T(1  +\|y\|)\int_0^{{+\infty}}e^{-\frac{\eta}{2} s}ds\\
&\leq&C_T(1  +\|y\|).
\end{eqnarray*}
\end{proof}
\end{lemma}

\subsection{\textbf{Determination of remainder} $ {r^\epsilon}$}
Once  $u_0$ and $u_1$   have been determined, we can carry out the
construction of the remainder $r^\epsilon$. It is known that
\begin{eqnarray*}
(\partial_t-\mathcal{L}^\epsilon)u^\epsilon=0,
\end{eqnarray*}
which, together with \eqref{u-o-equ-1} and \eqref{u-0-equ-2},
implies
\begin{eqnarray*}
(\partial_t-\mathcal{L}^\epsilon)r^\epsilon&=&-(\partial_t-\mathcal{L}^\epsilon)
u_0-\epsilon(\partial_t-\mathcal{L}^\epsilon)u_1\\
&=&-(\partial_t-\frac{1}{\epsilon}\mathcal{L}_1-\mathcal{L}_2)u_0-\epsilon(\partial_t-\frac{1}{\epsilon}\mathcal{L}_1-\mathcal{L}_2)u_1\\
&=&\epsilon(\mathcal{L}_2u_1-\partial_tu_1).
\end{eqnarray*}
In order to estimate the remainder term $r^\epsilon$  we need the
following crucial lemmas.

\begin{lemma}\label{4.3}
Assume Hypotheses 1  and 2. Then for any $x\in D(\mathcal{A})$,
$y\in D(A)$ and $T>0$, we have
\begin{eqnarray*}
\left|\frac{du_1}{dt}(t,x,y)\right|\leq C (1+\||x\||_1 )\|y\|,
\;t\in [0, T].
\end{eqnarray*}
\begin{proof}
According to \eqref{u-1}, we have
\begin{equation*}
\frac{du_1}{dt}(t,x,y)=\int_0^{{+\infty}}\mathbb{E}\left(\bar{\bf{F}}(x)-
\textbf{F}(x,Y_s(y)), \frac{d}{dt}D_x\bar{u}(t,x)\right)_{\mathcal
{H}}ds.
\end{equation*}
For any $h=(h_1,h_2)^T\in \mathcal{H}^1$,
\begin{eqnarray*}
D_x\bar{u}(t,x)\cdot h&=&D_x[\mathbb{E}\phi(\Pi_1\circ\bar{X}(t,x))]\\
&=&\mathbb{E}[\phi'(\Pi_1\circ\bar{X}(t,x))\cdot
(\Pi_1\circ\eta_t^{h,x})],
\end{eqnarray*}
and hence
\begin{eqnarray*}
\frac{d}{dt}(D_x\bar{u}(t,x)\cdot
h)&=&\mathbb{E}\left[\phi''(\Pi_1\circ\bar{X}(t,x))\cdot\left(\Pi_1\circ\eta_t^{h,x},
\frac{d}{dt}\Big(\Pi_1\circ\bar{X}(t,x)\Big) \right)\right]\\
&+&\mathbb{E}\left[\phi'(\Pi_1\circ\bar{X}(t,x))\cdot\frac{d}{dt}(\Pi_1\circ\eta_t^{h,x})\right],
\end{eqnarray*}
so that, due to the fact of $\phi\in C_b^2(H, \mathbb{R})$, we
obtain
\begin{eqnarray}
\left|\frac{d}{dt}(D_x\bar{u}(t,x)\cdot h)\right|&\leq&
C\left[\mathbb{E}\|\Pi\circ
\eta_t^{h,x}\|^2\right]^\frac{1}{2} \cdot\left[\mathbb{E}\|\frac{d}{dt}\Big(\Pi_1\circ\bar{X}(t,x)\|^2\right]^\frac{1}{2}\nonumber\\
&&+
\mathbb{E}\|\frac{d}{dt}(\Pi_1\circ\eta_t^{h,x}\|\nonumber\\
&\leq& C \||h\||_1\cdot
\left[\mathbb{E}\|\frac{d}{dt}\Big(\Pi_1\circ\bar{X}(t,x)\|^2\right]^\frac{1}{2}\nonumber\\
&&+
\mathbb{E}\|\frac{d}{dt}(\Pi_1\circ\eta_t^{h,x}\|,\label{derive-u-bar}
\end{eqnarray}
where we used the estimate \eqref{eta-1} such that
\begin{eqnarray*}
\|\Pi_1\circ\eta_t^{h,x}\| \leq C\||h\|| \leq C\||h\||_1.
\end{eqnarray*}
Now, as $\bar{X}_t(x)$ is the  mild solution of averaging equation
with initial data $x=(x_1,x_2)^T\in \mathcal{H}^1,$ we have
\begin{eqnarray*}
\Pi_1\circ \bar{X}_t(x)&=&\bar{U}_t(x)\\
&=&C(t) x_1+(-A)^{-\frac{1}{2}}S(t) x_2+\int_0^t
(-A)^{-\frac{1}{2}}S(t-s)\bar{F}(\bar{U}_s(x))ds\\
&&+\sigma_1\int_0^t (-A)^{-\frac{1}{2}}S(t-s)dW^1_s\nonumber
\end{eqnarray*}
with
\begin{eqnarray*}
\frac{d}{dt}[\Pi_1\circ \bar{X}_t(x)]&=&-(-A)^\frac{1}{2}S(t)
x_1+C(t)
x_2+\int_0^t C(t-s)\bar{F}(\bar{U}_s(x))ds\\
&&+\sigma_1\int_0^t
 C(t-s)dW^1_s,
\end{eqnarray*}
By straightforward computation, we have
\begin{eqnarray}
\|-(-A)^\frac{1}{2}S(t) x_1\|^2\leq \|x_1\|_1^2 \label{4-11}
\end{eqnarray}
and
\begin{eqnarray}
\|C(t) x_2\|^2\leq \|x_2\|^2.\label{4-12}
\end{eqnarray}
According to the Lipschitz continuity of $\bar{F}$ and
\eqref{bar-X-bound}, we have
\begin{eqnarray}
\mathbb{E}\|\int_0^t C(t-s)\bar{F}(\bar{U}_s(x))ds\|^2&\leq&
C_T\mathbb{E}\int_0^t(1+\|\bar{U}_s(x)\|^2)ds\nonumber\\
&\leq& C_T (1+\||x\||_1). \label{4-13}
\end{eqnarray}
Now, from \eqref{4-11}-\eqref{4-13} it follows
\begin{eqnarray}
\mathbb{E}\|\frac{d}{dt}[\Pi_1\circ \bar{X}_t(x)]\|^2\leq
C(1+\||x\||^2_1) \label{4-13-1}
\end{eqnarray}
Now, we prove uniform bounds for time derivative of
$\Pi_1\circ\eta_t^{h,x}$ with respect to $x$. Clearly,  we have
\begin{eqnarray*}
 \frac{d}{dt}(\Pi_1\circ\eta_t^{h,x})&=&-(-A)^\frac{1}{2}S(t)h_1+C(t) h_2\\
 &+&\int_0^t
 C(t-s)[\Pi_1\circ(\bar{\textbf{F}}'(\bar{X}_s(x))\cdot\eta_t^{h,x})]ds.
\end{eqnarray*}
Note that for any $h=(h_1, h_2)^T\in \mathcal{H}^1$,
\begin{eqnarray}
\|(-A)^\frac{1}{2}S(t)h_1\|^2+\|C(t) h_2\|^2\leq \||h\||_1^2.
\label{4-14}
\end{eqnarray}
In view of \eqref{6.4} and \eqref{eta-1}, we obtain
\begin{eqnarray}
&&\|\int_0^t
 C(t-s)[\Pi_1\circ(\bar{\textbf{F}}'(\bar{X}_s(x))\cdot\eta_t^{h,x})]ds\|\nonumber\\
 &\leq& C\int_0^t\||\eta_s^{h,x}\||ds
 \nonumber\\
 &\leq& C_T \||h\||_1\label{4-15}.
\end{eqnarray}
Then thanks to \eqref{4-14} and \eqref{4-15}, we obtain
\begin{eqnarray*}
\|\frac{d}{dt}(\Pi_1\circ\eta_t^{h,x})\|^2\leq C \||h\||_1^2.
\end{eqnarray*}
So, if we plug the above estimate and estimate \eqref{4-13-1} into
\eqref{derive-u-bar}, we get
\begin{eqnarray*}
\left|\frac{d}{dt}D_x\bar{u}(t,x)\cdot h\right|\leq
C\||h\||_1(1+\||x\||_1 ),
\end{eqnarray*}
which, together with \eqref{Averaging-Expectation}, implies
\begin{eqnarray*}
\left|\frac{du_1}{dt}(t,x,y)\right|&\leq&C(1+\||x\||_1
)\int_0^{{+\infty}}\||\bar{\bf{F}}(x)- \mathbb{E}\textbf{F}(x,Y_s(y))\||_1ds\\
&\leq&C (1+\||x\||_1)\|y\|\int_0^{{+\infty}}e^{-\frac{\eta}{2} s}ds\\
&\leq&C (1+\||x\||_1 )\|y\|.
\end{eqnarray*}
Hence the assertions is completely proved.

\end{proof}
\end{lemma}

\begin{lemma}\label{4.4}
Assume that all conditions in Lemma \ref{4.3}  are fulfilled. Then
we have
\begin{eqnarray*}
\left|\mathcal {L}_2u_1(t,x,y)\right|\leq
\big(1+\||\mathcal{A}x\||+\||x\||_1+\|y\| \big)\big(1+\|y\| \big),
\;t\in [0, T].
\end{eqnarray*}
\begin{proof}
As known, for any $x\in D({\mathcal{A}})$,
\begin{eqnarray*}
\mathcal {L}_2u_1(t,x,y)&=&\Big(\mathcal{A}x+ \textbf{F}(x, y),
D_xu_1(t,x,y)
\Big)_\mathcal{H}\\
&+&\frac{1}{2}\sigma_2^2Tr\Big(D^2_{xx}u_1(t,x,y)(BQ_1)(BQ_1)^*\Big).
\end{eqnarray*}
We will carry out the estimate of $\left|\mathcal {L}_2u_1(t,x,y)\right|$ in two steps.\\
(\textbf{Step 1}) Estimate of $\Big(\mathcal{A}x+ \textbf{F}(x, y),
D_xu_1(t,x,y) \Big)_\mathcal{H}$.

 For any $k\in \mathcal{H}$, we have
\begin{eqnarray*}
D_xu_1(t,x,y)\cdot
k&=&\int_0^{{+\infty}}\mathbb{E}\Big(D_x(\bar{\bf{F}}(x)-
\textbf{F}(x,Y_s))\cdot k,D_x\bar{u}(t,x)\Big)_{\mathcal{H} }ds\\
&&+\int_0^{{+\infty}}\mathbb{E}\Big(\bar{\bf{F}}(x)-
\textbf{F}(x,Y_s),D^2_{xx}\bar{u}(t,x)\cdot
k\Big)_{\mathcal{H} }ds\\
&:=&I_1(t,x,y,k)+I_2(t,x,y,k).
\end{eqnarray*}
According to the invariant property of measure $\mu$, \eqref{6.3}
and \eqref{F'-lip} we have
\begin{eqnarray}
&&\!\!\!\!\!\!\!\!\!\!|I_1(t,x,y,k)|\nonumber\\
&\leq& \int_0^{+\infty} \left|\mathbb{E} \Big(D_x(\bar{\bf{F}}(x)-
\textbf{F}(x,Y_s(y)))\cdot
k,D_x\bar{u}(t,x)\Big)_{\mathcal{H} }       \right| ds\nonumber\\
&=&\int_0^{+\infty}   \left|  \mathbb{E} \int_H
\Big(D_x[\textbf{F}(x,z)- \textbf{F}(x,Y_s(y))]\cdot
k,D_x\bar{u}(t,x)\Big)_{\mathcal{H} }\mu(dz)\right|ds\nonumber\\
&=&\int_0^{+\infty}   \left|  \mathbb{E} \int_H
\Big(D_x[\textbf{F}(x,Y_s(z))- \textbf{F}(x,Y_s(y))]\cdot
k,D_x\bar{u}(t,x)\Big)_{\mathcal{H} }\mu(dz)\right|ds\nonumber\\
&\leq&C\||k\|| \cdot\||D_x\bar{u}(t,x)\||
\cdot\int_0^{+\infty}\left[\int_H\mathbb{E}\|Y_s(z)-Y_s(y)\|\mu(dz)\right]ds\nonumber
\end{eqnarray}
By making use of \eqref{initial-diff} and \eqref{mu-Momenent-bound},
the above yields
\begin{eqnarray}
|I_1(t,x,y,k)|&\leq&C\||k\|| \cdot\||D_x\bar{u}(t,x)\||
\cdot\int_0^{+\infty}
e^{-\frac{\eta}{2} s} (1+\|y\|)ds\nonumber\\
&\leq& C\||k\|| \cdot\||D_x\bar{u}(t,x)\||\nonumber \\
&\leq& C\||k\||, \label{I_1}
\end{eqnarray}
where we used  Lemma \ref{ux} in the last step. By Lemma \ref{uxx}
and \eqref{Averaging-Expectation},  we have
\begin{eqnarray}
|I_2(t,x,y,k)|&\leq& \int_0^{{+\infty}}\left|\Big(\bar{\bf{F}}(x)-
\mathbb{E}\textbf{F}(x,Y_s(y)),D^2_{xx}\bar{u}(t,x)\cdot
k\Big)_{\mathcal{H}
}\right|ds\nonumber\\
&\leq& C\||k\|| \int_0^{{+\infty}}\||\bar{\bf{F}}(x)-
\mathbb{E}\textbf{F}(x,Y_s(y))\||  ds\nonumber\\
&\leq&C\||k\|| (1  +\|y\|)\int_0^{+\infty} e^{-\frac{\eta}{2} s} ds\nonumber\\
&\leq&C\||k\|| (1  +\|y\|).\nonumber
\end{eqnarray}
Together with \eqref{I_1} , this yields
\begin{eqnarray*}
\left|D_xu_1(t,x,y)\cdot k\right|&\leq&C\||k\|| (1+ \|y\|)
\end{eqnarray*}
which means
\begin{eqnarray}
&&\left|\Big( \mathcal{A}x+ \textbf{F}(x,
y), D_xu_1(t,x,y) \Big)_\mathcal{H}\right|\nonumber\\
&&\leq C\big(1+\||\mathcal{A}x\||+\||x\||_1+\|y\| \big)\big(1+\|y\|
\big). \label{4.18}
\end{eqnarray}
(\textbf{Step 2}) Estimate of
$Tr\Big(D^2_{xx}u_1(t,x,y)(BQ_1)(BQ_1)^*\Big)$. Note that we have
\begin{eqnarray*}
&&D_{xx}u_1(t,x,y)\cdot (h,
k)\\
&&=\int_0^{{+\infty}}\mathbb{E}\Big(D^2_{xx}(\bar{\bf{F}}(x)-
\textbf{F}(x,Y_s(y)))\cdot (h,k),D_x\bar{u}(t,x)\Big)_{\mathcal{H} }ds\\
&&+\int_0^{{+\infty}}\mathbb{E}\Big(D_x(\bar{\bf{F}}(x)-
\textbf{F}(x,Y_s(y)))\cdot h,D^2_{xx}\bar{u}(t,x)\cdot
k\Big)_{\mathcal{H} }ds\\
&&+\int_0^{{+\infty}}\mathbb{E}\Big(D_x(\bar{\bf{F}}(x)-
\textbf{F}(x,Y_s(y)))\cdot k,D^2_{xx}\bar{u}(t,x)\cdot
h\Big)_{\mathcal{H} }ds\\
&&+\int_0^{{+\infty}}\mathbb{E}\Big(\bar{\bf{F}}(x)-
\textbf{F}(x,Y_s(y)), D^3_{xxx}\bar{u}(t,x)\cdot
(h,k)\Big)_{\mathcal{H}
}ds\\
&&:=\sum\limits_{i=1}^4J_i(t,x,y,h,k).
\end{eqnarray*}
In view of \eqref{regular-4} and invariant property of measure $\mu$
we have
\begin{eqnarray}
&&\!\!\!\!\!\!\!\!\!\!|J_1(t,x,y,h,k)|\nonumber\\
&\leq& \int_0^{+\infty} \left|\mathbb{E}
\Big(D^2_{xx}(\bar{\bf{F}}(x)- \textbf{F}(x,Y_s(y)))\cdot
(h,k),D^2_{x}\bar{u}(t,x)\Big)_{\mathcal{H} }       \right| ds\nonumber\\
&=&\int_0^{+\infty}   \left|  \mathbb{E} \int_H
\Big(D^2_{xx}[\textbf{F}(x,z)- \textbf{F}(x,Y_s(y))]\cdot
(h,k),D_x\bar{u}(t,x)\Big)_{\mathcal{H} }\mu(dz)\right|ds\nonumber\\
&=&\int_0^{+\infty}   \left|  \mathbb{E} \int_H
\Big(D^2_{xx}[\textbf{F}(x,Y_s(z))- \textbf{F}(x,Y_s(y))]\cdot
(h,k),D_x\bar{u}(t,x)\Big)_{\mathcal{H} }\mu(dz)\right|ds\nonumber
\end{eqnarray}
By taking \eqref{F''-lip} and Lemma \ref{ux}   into account, we can
deduce
\begin{eqnarray}
&&\!\!\!\!\!\!\!\!\!\!|J_1(t,x,y,h,k)|\nonumber\\
 &\leq&C\||h\||\cdot\||k\||
\cdot\int_0^{+\infty}\left[\int_H\mathbb{E}\|Y_s(z)-Y_s(y)\| \mu(dz)\right]ds\nonumber\\
&\leq&C\||h\||\cdot\||k\||  \cdot\int_0^{+\infty}
e^{-\frac{\eta}{2} s} (1+\|y\|)ds\nonumber\\
&\leq& C\||h\||\cdot\||k\||  \nonumber \\
&\leq& C\||h\||\cdot\||k\|| (1+\|y\|), \label{J_1}
\end{eqnarray}
Again, by   \eqref{6.3} and invariant property of measure $\mu$, we
have
\begin{eqnarray}
&&\!\!\!\!\!\!\!\!\!\!|J_2(t,x,y,h,k)|\nonumber\\
&\leq& \int_0^{+\infty}   \left|  \mathbb{E} \int_H
\Big(D_{x}[\textbf{F}(x,Y_s(z))- \textbf{F}(x,Y_s(y))]\cdot
h,D^2_{xx}\bar{u}(t,x)\cdot k\Big)_{\mathcal{H}
}\mu(dz)\right|ds,\nonumber
\end{eqnarray}
which, by Lemma \ref{uxx} and condition \eqref{F'-lip}, implies
\begin{eqnarray}
&&\!\!\!\!\!\!\!\!\!\!|J_2(t,x,y,h,k)|\nonumber\\
&\leq&C\||h\||\cdot\||k\||
\cdot\int_0^{+\infty}\left[\int_H\mathbb{E}\|Y_s(z)-Y_s(y)\|\mu(dz)\right]ds\nonumber\\
&\leq&C\||h\||\cdot\||k\|| \cdot\int_0^{+\infty}
e^{-\frac{\eta}{2} s} (1+\|y\|)ds\nonumber\\
&\leq& C\||h\||\cdot\||k\||(1+\|y\|). \label{J_2}
\end{eqnarray}
Parallel to \eqref{J_2},  we can obtain the same estimate for
$J_3(t,x,y, h,k)$, that is,
\begin{eqnarray}
|J_3(t,x,y,h,k)| \leq C\||h\||\cdot\||k\||(1+\|y\|). \label{J_3}
\end{eqnarray}
By proceeding again as in the estimate for $J_1(t,x,y,h,k)$ we have
\begin{eqnarray}
&&\!\!\!\!\!\!\!\!\!\!|J_4(t,x,y,h,k)|\nonumber\\
&\leq& \int_0^{+\infty}   \left|  \mathbb{E} \int_H
\Big(\textbf{F}(x,Y_s(z))- \textbf{F}(x,Y_s(y))
,D^3_{xxx}\bar{u}(t,x)\cdot(h,k)\Big)_{\mathcal{H}
}\mu(dz)\right|ds\nonumber
\end{eqnarray}
and then thanks to Lemma \ref{uxxx} and \eqref{F-condi-1}, we get
\begin{eqnarray}
&&\!\!\!\!\!\!\!\!\!\!|J_4(t,x,y,h,k)|\nonumber\\
&\leq&C\||h\||\cdot\||k\||
\cdot\int_0^{+\infty}\left[\int_H\mathbb{E}\|Y_s(z)-Y_s(y)\| \mu(dz)\right]ds\nonumber\\
&\leq& C\||h\||\cdot\||k\||(1+\|y\|). \label{J_4}
\end{eqnarray}
Collecting together \eqref{J_1}, \eqref{J_2}, \eqref{J_3} and
\eqref{J_4}, we obtain

\begin{eqnarray*}
|D^2_{xx}u_1(t,x,y)\cdot (h, k)|\leq C \||h\||\cdot\||k\||(1+\|y\|),
\end{eqnarray*}
which means that for fixed $y\in H$ and $t\in [0, T]$,
\begin{eqnarray*}
\|D^2_{\cdot\cdot}u_1(t,\cdot,y)
\|_{L(\mathcal{H}\times\mathcal{H},\mathbb{R})}\leq C(1+\|y\|),
\end{eqnarray*}
so that, as the operator $Q_1$ has finite trace, we get
\begin{eqnarray}
&&Tr\Big(D^2_{xx}u_1(t,x,y)(BQ_1)(BQ_1)^*\Big)\nonumber\\
&&\leq \|D^2_{xx}u_1(t,x,y)\|Tr\Big( (BQ^1)(BQ^1)^*\Big)\nonumber\\
&&\leq C(1+\|y\|). \label{Tr}
\end{eqnarray}

Finally, by taking  inequalities \eqref{4.18} and \eqref{Tr} into
account, we can conclude the proof of the lemma.
\end{proof}
\end{lemma}
As a consequence of Lemma \ref{4.3} and \ref{4.4}, we have the
following fact for the remainder term $r^\epsilon$.
\begin{lemma}\label{4.5}
Under the conditions of Lemma \ref{4.3}, for any $T>0$, $x\in
D(\mathcal{A}), y \in H$, we have
\begin{eqnarray*}
|r^\epsilon(T,x,y)|\leq
C\epsilon(1+\||x\||+\|y\|)(1+\||\mathcal{A}x\||+\||x\||_1).
\end{eqnarray*}

\begin{proof}
By a variation of constant formula, we have
\begin{eqnarray*}
r^\epsilon(T,x,y)&=&\mathbb{E}[r^\epsilon(0,X^\epsilon_T(x,y),Y_{T/\epsilon}( y)]\\
&+&\epsilon\mathbb{E}\left[\int_0^T(\mathcal{L}_2u_1-\frac{\partial
u_1}{\partial s})( X^\epsilon_{T-s}(x,y),Y_{\frac{T-s}{\epsilon}}(
y)) ds\right].
\end{eqnarray*}
Since $u^\epsilon$ and $u_0=\bar{u}$ has the same initial condition
$\Phi(x)$, we have
\begin{eqnarray*}
|r^\epsilon(0, x,y)|&=&|u^\epsilon(0,x,y)-\bar{u}(0,x)-\epsilon
u_1(0,x,y)|\\
&=&\epsilon |u_1(0,x,y)|,
\end{eqnarray*}
so that, from \eqref{u-1-abso} and \eqref{Y-bound}   we have
\begin{eqnarray}
\mathbb{E}[r^\epsilon(0,X^\epsilon_T(x,y),Y_{T/\epsilon}(y)]\leq
C\epsilon(1+ \|y\|). \label{r-bound}
\end{eqnarray}
Thanks to Lemma \ref{4.3} and Lemma \ref{4.4}, we have
\begin{eqnarray*}
&&\mathbb{E}[(\mathcal{L}_2u_1-\frac{\partial u_1}{\partial s})(
X^\epsilon_{T-s}(x,y),Y_{\frac{T-s}{\epsilon}}(y))]\\
&&\leq C \mathbb{E}\big[(1+\||\mathcal{A}X^\epsilon_{T-s}(x,y)\||
+\||X^\epsilon_{T-s}(x,y)\||_1+\|Y_{\frac{T-s}{\epsilon}}(y)\|)\nonumber\\
&&\quad\cdot(1+\|Y_{\frac{T-s}{\epsilon}}(y)\|)\big],
\end{eqnarray*}
and, according to \eqref{Y-bound}, \eqref{2.6}, \eqref{AX-bound} and
the H\"{o}lder inequality, this implies that
\begin{eqnarray*}
&&\mathbb{E}\left[\int_0^T(\mathcal{L}_2u_1-\frac{\partial
u_1}{\partial s})( X^\epsilon_{T-s}(x,y),Y_{\frac{T-s}{\epsilon}})
ds\right]\nonumber\\
&&\leq C (1+ \|y\|)(1+\||\mathcal{A}x\||+\||x\||_1+\|y\|).
\end{eqnarray*}
This, together with \eqref{r-bound}, implies
\begin{eqnarray*}
|r^\epsilon(T,x,y)|\leq C\epsilon(1+ \|y\|)
(1+\||\mathcal{A}x\||+\||x\||_1+\|y\|),
\end{eqnarray*}
which completes the proof.

\end{proof}

\end{lemma}

\section{Appendix}
In this section, we state and prove some technical lemmas needed in
the former sections.
\begin{lemma}
For any $x\in \mathcal{H}$ and $y\in H$, there exists a constant
$C>0$ such that
\begin{eqnarray}
\||\bar{\bf{F}}(x)-\mathbb{E}[\textbf{F}(x,Y_t(y))]\||^2_1 \leq
Ce^{-\eta t}\Big(1 +\|y\|^2\Big),\label{Averaging-Expectation}
\end{eqnarray}
where $\eta= \alpha_1-L_g >0.$
\begin{proof}
According to the invariant property of $\mu $,
\eqref{mu-Momenent-bound} and hypothesis \eqref{F-condi-1}, we have
\begin{eqnarray*}
\||\bar{\bf{F}}(x)-\mathbb{E}[\textbf{F}(x,Y_t(y))]\||^2_1
&=& \|\bar{ {F}}(\Pi_1x)- \mathbb{E}[{F}(\Pi_1x,Y_t(y))]\|^2\\
&=&\|\int_HF(\Pi_1 x,z)\mu(dz)- \mathbb{E}[{F}(\Pi_1x,Y_t(y))]\|^2\\
&=&\|\int_H \mathbb{E}[F(\Pi_1x,
Y_t(z))-F(\Pi_1x,Y_t(y))]\mu(dz)\|^2,
\end{eqnarray*}
so thanks to \eqref{mu-Momenent-bound} and \eqref{initial-diff}, we
have
\begin{eqnarray*}
\||\bar{\textbf{F}}(x)-\mathbb{E}[\textbf{F}(x,Y_t(y))]\||_1^2&\leq&
C\int_H\mathbb{E}\left\|Y_t(y)-Y_t(z)\right\|^2\mu (dz)\nonumber\\
\nonumber&\leq&Ce^{-\eta t}\int_H\|y-z\|^2\mu (dz)\\
&\leq&Ce^{-\eta t}\Big(1 +\|y\|^2\Big).
\end{eqnarray*}
\end{proof}
\end{lemma}

Next, we introduce the following regularity results of averaging
function $\bar{F}$.
\begin{lemma}\label{regular-lemma}
For any $w\in H$, the function $\Big(\bar{F}(\cdot), w\Big)_H:\;
H\rightarrow \mathbb{R}$ is G\^{a}teaux differential and for any
$v\in H$, it hold that
\begin{eqnarray*}
\Big(D\bar{F}(u)\cdot v, w\Big)_H=\int_H\Big(D_u F(u,y)\cdot v,
w\Big)_H\mu(dy),\;w\in H.
\end{eqnarray*}
\begin{proof}
For any $\lambda \neq0$ we have
\begin{eqnarray}
&&\Big(\int_H\frac{1}{\lambda}[F(u+\lambda
v,y)-F(u,y)]\mu(dy),w\Big)_H-\int_H\Big(D_uF(u,y)\cdot
v,w\Big)_H\mu(dy)\nonumber\\
&&=\int_H \Big(\frac{1}{\lambda}[F(u+\lambda
v,y)-F(u,y)]-D_uF(u,y)\cdot v,w\Big)_H\mu(dy).\nonumber
\end{eqnarray}
and then
\begin{eqnarray}
&&\!\!\!\left|\Big(\int_H\frac{1}{\lambda}[F(u+\lambda
v,y)-F(u,y)]\mu(dy),w\Big)_H-\int_H\Big(D_uF(u,y)\cdot
v,w\Big)_H\mu(dy)\right|\nonumber\\
&&\!\!\!\leq\int_H \left|\Big(\frac{1}{\lambda}[F(u+\lambda
v,y)-F(u,y)]-D_uF(u,y)\cdot v,w\Big)_H\right|\mu(dy)\nonumber\\
&&\!\!\!\leq\|w\|\int_H  \|\frac{1}{\lambda}[F(u+\lambda
v,y)-F(u,y)]-D_uF(u,y)\cdot v\| \mu(dy).\label{regular-1}
\end{eqnarray}
Now, since $F(\cdot, y): H\rightarrow H$ is G\^{a}teaux
differentiable in $H$, for any $h\in H$, we obtain
\begin{eqnarray}
\lim\limits_{\lambda\rightarrow 0}\|\frac{1}{\lambda}[F(u+\lambda
v,y)-F(u,y)]-D_uF(u,y)\cdot v\|=0.\label{regular-2}
\end{eqnarray}
Moreover, by mean value theorem,
\begin{eqnarray}
&&\frac{1}{\lambda}[F(u+\lambda v,y)-F(u,y)]-D_uF(u,y)\cdot
v\nonumber\\
&&=\int_0^1[D_u F(u+\lambda \theta v)-D_uF(u,y)]\cdot v
d\theta\nonumber
\end{eqnarray}
so that, due to the boundedness of $D_uF(u,y)$, we get
\begin{eqnarray*}
\|\frac{1}{\lambda}[F(u+\lambda v,y)-F(u,y)]-D_uF(u,y)\cdot v\|\leq
C \|v\|.
\end{eqnarray*}
and then, by using the dominated convergence theorem, taking
\eqref{regular-2} into account, we can conclude
\begin{eqnarray}
&&\lim\limits_{\lambda\rightarrow
0}\Big(\int_H\frac{1}{\lambda}[F(u+\lambda
v,y)-F(u,y)]\mu(dy),w\Big)_H\nonumber\\
&&=\int_H\Big(D_uF(u,y)\cdot v,w\Big)_H\mu(dy),\nonumber
\end{eqnarray}
which implies that
\begin{eqnarray*}
\Big(D\bar{F}(u)\cdot v, w\Big)_H=\int_H\Big(D_u F(u,y)\cdot
v,w\Big)_H\mu(dy).
\end{eqnarray*}
\end{proof}
\end{lemma}

\begin{remark}
As a as a consequence of Lemma \ref{regular-lemma}, it is easily to
check that
\begin{eqnarray}
\Big(D_x\bar{\bf{F}}(x)\cdot  h, k\Big)_\mathcal{H}=\int_H\Big(D_x
\textbf{F}(x,y)\cdot h,k\Big)_\mathcal{H}\mu(dy), \;h,k\in
\mathcal{H}, \label{6.3}
\end{eqnarray}
and this yields
\begin{eqnarray}
\||D_x\bar{\bf{F}}(x)\cdot h\|| &\leq& \int_H\||D_x
\textbf{F}(x,y)\cdot h\||\mu(dy)\nonumber\\
&=&\int_H\|A^{-\frac{1}{2}}[(D_uF)(\Pi_1\circ x,y)\cdot(\Pi_1\circ
h)]\| \mu(dy)\nonumber\\
&\leq&C\int_H\| (D_uF)(\Pi_1\circ x,y)\cdot(\Pi_1\circ h) \|
\mu(dy).\nonumber
\end{eqnarray}
Moreover, by invoking conditions \eqref{F'}, we get
\begin{eqnarray}
\||D_x\bar{\bf{F}}(x)\cdot h\|| &\leq&C \||h\||,\; h\in
\mathcal{H}.\label{6.4}
\end{eqnarray}

 As far as the higher order derivative are concerned, by proceeding
as in the proof of above lemma, we can show that
\begin{eqnarray*}
\Big(D_{uu}^2\bar{F}(u)\cdot(v,w), \nu\Big)_H=\int_H\Big(D_{uu}
F(u,y)\cdot(v,w),\nu\Big)_H\mu(dy), \;v,w,\nu\in H.\label{regular-3}
\end{eqnarray*}
As a consequence, we obtain
\begin{eqnarray}
\Big(D_{xx}^2\bar{\bf{F}}(x)\cdot(h,k), l\Big)_\mathcal
{H}=\int_H\Big(D_{xx}^2 {\bf{F}}(x,y)\cdot(h,k),l\Big)_\mathcal
{H}\mu(dy), \;h,k,l\in \mathcal {H} \label{regular-4}
\end{eqnarray}
and
\begin{eqnarray}\label{5.7}
\||D_{xx}^2\bar{\bf{F}}(x)\cdot(h,k)\|\leq
C\||h\||\cdot\||k\||,\;h,k\in \mathcal{H}.
\end{eqnarray}
\end{remark}

\begin{lemma}\label{ux}
For any $T>0$, there exists $C_T>0$ such that for any
$x\in\mathcal{H}$ and $t\in [0, T]$, we have
$$\|D_{x}\bar{u}(t,x)\|\leq C_{T}.$$
\end{lemma}
\begin{proof}
Note that for any $h\in \mathcal{H}$,
\begin{eqnarray*}
D_x\bar{u}(t,x)\cdot
h&=&\mathbb{E}\left[D\Phi(\bar{X}_t(x))\cdot\eta ^{h,x}_t\right]\\
&=&\mathbb{E}\left(\Phi'(\bar{X}_t(x)),\eta^{h,x}_t\right)_{\mathcal{H}},
\end{eqnarray*}
where $\eta^{h,x}_t$ is the mild solution of
\begin{eqnarray*}
\begin{cases}
d\eta ^{h,x}_t=\left(\mathcal{A}\eta ^{h,x}_t+D\bar{\bf{F}}(\bar{X}_t(x))\cdot\eta ^{h,x}_t\right)dt\\
\eta ^{h,x}(0)=h.
\end{cases}
\end{eqnarray*}
This means that  $\eta ^{h,x}_t$  is the  solution of the integral
equation
\begin{eqnarray*}
\eta^{h,x}_t=\mathcal{S}_th+\int_0^t\mathcal{S}_{t-s}[D\bar{\bf{F}}(\bar{X}_s(x))\cdot\eta
^{h,x}_t]ds,
\end{eqnarray*}
and then thanks to \eqref{6.4}, we get
\begin{eqnarray*}
\||\eta ^{h,x}_t\|| \leq \||h\|| +C\int_0^t\||\eta ^{h,x}_s\|| ds.
\end{eqnarray*}
Then by  Gronwall lemma it follows that
\begin{eqnarray}
\||\eta ^{h,x}_t\|| \leq C_T\||h\||, \;t\in [0, T],  \label{eta-1}
\end{eqnarray}
which means
\begin{eqnarray*}
|D_x\bar{u}(t,x)\cdot h|\leq C_T\sup\limits_{z\in \mathcal{H}
}|\Phi(z)|\cdot\||h\||, \label{u-bar-deriv}
\end{eqnarray*}
so that
\begin{eqnarray*}
\||D_x\bar{u}(t,x)\|| \leq C_T.\label{u-bar-deriv-1-1}
\end{eqnarray*}
\end{proof}
\begin{lemma}\label{uxx}
For any $T>0$, there exists $C_T>0$ such that for any
$x,h,k\in\mathcal{H}$ and $t\in [0, T]$, we have
$$\left|D^2_{xx}\bar{u}(t,x)\cdot(h,k)\right|\leq C_{T,\phi}\||h\||\cdot\||k\||.$$

\begin{proof}
For any $h, k \in \mathcal{H}$, we have
\begin{eqnarray}
D^2_{xx}\bar{u}(t,x)\cdot(h,k)&=&\mathbb{E}\big[\Phi''(\bar{X}_t(x))\cdot(\eta^{h,x}_t,\eta^{k,x}_t)\nonumber\\
&+&\Phi'(\bar{X}_t(x))\cdot \zeta^{h,k,x}_t\big], \label{5.1}
\end{eqnarray}
where $\zeta^{h,k,x}$ is the mild solution of equation
\begin{eqnarray*}
\begin{cases}
d\zeta^{h,k,x}_t=\left[\mathcal{A}\zeta^{h,k,x}_t+ \bar{\bf{F}}''(\bar{X}_t(x))\cdot(\eta^{h,x}_t,\eta^{k,x}_t)+\bar{\bf{F}}'(\bar{X}_t(x))\cdot\zeta^{h,k,x}_t\right]dt\\
\zeta^{h,k,x}_0=0.
\end{cases}
\end{eqnarray*}
This means that  $\zeta^{h,k,x}_t$  is the  solution of the integral
equation
\begin{eqnarray*}
\zeta^{h,k,x}_t=
\int_0^t\mathcal{S}_t\big[\bar{\bf{F}}''(\bar{X}_s(x))\cdot(\eta^{h,x}_s,\eta^{k,x}_s)+\bar{\bf{F}}'(\bar{X}_s(x))\cdot\zeta^{h,k,x}_s\big]ds.
\end{eqnarray*}
Thus, by \eqref{6.4} and \eqref{5.7} we have
\begin{eqnarray*}
\||\zeta^{h,k,x}_t\||&\leq& C\int_0^t(\||\eta^{h,x}_s\||\cdot
\||\eta^{k,x}_s\||+\||\zeta^{h,k,x}_s\||)ds\\
&\leq& C\||h\||\cdot\||k\|| +C\int_0^t\||\zeta^{h,k,x}_s\||ds.
\end{eqnarray*}
By applying the Gronwall lemma we have
\begin{equation*}
\||\zeta^{h,k,x}_t\||\leq C_T\||h\||\cdot \||k\||, \; t>0.
\end{equation*}
Returning to \eqref{5.1}, we can get
\begin{eqnarray}
|D_{xx}\bar{u}(t,x)\cdot(h,k)|\leq C \||h\||\cdot\||k\||.
\end{eqnarray}
\end{proof}
\end{lemma}
By proceeding again as in the proof of above lemma,  we have the
following result.
\begin{lemma}\label{uxxx}
For any $T>0$, there exists $C_T>0$ such that for any
$x,h,k,l\in\mathcal{H}$ and $t\in [0, T]$, we have
$$D^3_{xxx}\bar{u}(t,x)\cdot(h,k,l)\leq C_{T,\phi}\||h\||\cdot\||k\||\cdot\||l\||.$$
\end{lemma}

\section*{Acknowledgments}
We would like to thank Professor Jinqiao Duan for helpful
discussions and comments. Hongbo Fu is supported by CSC scholarship
(No. [2015]5104), NSF  of China (Nos. 11301403, 11405118, 11271295)
and Foundation of Wuhan Textile University 2013. Li Wan is supported
by NSF  of China (No. 11271295) and Science and Technology Research
Projects of Hubei Provincial Department of Education (No.
D20131602). Jicheng Liu is supported by NSF  of China (Nos.
11271013, 10901065). Xianming Liu is  supported by NSF  of China
(No. 11301197)

\label{}










\end{document}